%% file: Burgers.tex
\documentclass{article}
\usepackage[a4paper,left=1.2in,right=1.3in,top=1.1in,bottom=1.1in]{geometry}
\usepackage{amssymb,amsmath}
\usepackage{mhequ,mhenvs}
\usepackage{graphicx,tikz}
\usepackage{hyperref}

\def\C{\mathbb{C}}
\def\CD{\mathcal{D}}
\def\CF{\mathcal{F}}
\def\CH{\mathcal{H}}
\def\CO{\mathcal{O}}
\let\d=\partial
\DeclareMathOperator{\diag}{diag}
\def\DT{{\Delta t}}
\def\DX{{\Delta x}}
\def\e{\mathrm{e}}
\def\E{\mathbb{E}}
\let\eps=\varepsilon
\def\eqdef{\stackrel{\mbox{\tiny def}}{=}}
\def\eref#1{(\ref{#1})}

\let\la=\langle

\def\quark{\setbox0\hbox{$x$}\hbox to\wd0{\hss$\cdot$\hss}}
\def\R{\mathbb{R}}
\let\ra=\rangle
\def\Z{\mathbb{Z}}

\newtheorem{conjecture}{Conjecture}

\title{Approximations to the Stochastic Burgers Equation}
\author{Martin Hairer, Jochen Voss}
\date{24th May 2010}

\begin{document}

\maketitle

\begin{abstract}
  This article is devoted to the numerical study of various finite
  difference approximations to the stochastic Burgers equation. Of
  particular interest in the one-dimensional case is the situation
  where the driving noise is white both in space and in time.  We
  demonstrate that in this case, different finite difference schemes
  converge to different limiting processes as the mesh size tends to
  zero. A theoretical explanation of this phenomenon is given and we
  formulate a number of conjectures for more general classes of
  equations, supported by numerical evidence.\\[1em]
\end{abstract}

\section{Introduction}

This article studies several finite
difference schemes for the viscous stochastic Burgers equation:
\begin{equ}[e:BurgersBasic]
  \d_t u(x,t)
    = \nu\,\d_x^2 u - g(u)\,\d_x u + \sigma\,\xi(x,t),
    \qquad x\in [0,2\pi],\quad t\ge 0.
\end{equ}
In this equation, $\xi$ denotes space-time white noise, that is the
centred, distribution-valued Gaussian random variable such that
$\E\bigl(\xi(x,t) \xi(y,s)\bigr) = \delta(t-s) \delta(x-y)$. We will
always endow this equation with periodic boundary conditions and we
will consider solutions $u$ taking values either in $\R$ or in $\R^n$
(in which case $g$ is matrix-valued in general).

Motivations for studying the stochastic Burgers equation are manifold.
Just to name a few, it is used to model vortex lines in
high-temperature superconductors \cite{BFGL}, dislocations in
disordered solids and kinetic roughening of interfaces in epitaxial
growth \cite{Bar96}, formation of large-scale structures in the
universe \cite{GSS85,SZ89}, constructive quantum field theory
\cite{BCJ94}, \textit{etc}.  Since in the case $\sigma = 0$ and $g(u)
\propto u$ this equation is furthermore explicitly solvable via the
Hopf-Cole transform ($u = \d_x v/v$, where $v$ solves the heat equation),
it comes as no surprise that a wealth of numerical and analytical
results are available.  From a purely mathematical point of view, let
us mention for example the well-posedness results from
\cite{BCF91,BCJ94,DDT94,Gyo98,Kim06} and the ergodicity results
obtained in \cite{TwaZab,GoldMasl}.  One remarkable achievement was
the construction of a stationary solution in the inviscid limit with
non-vanishing noise \cite{EKhaSin00} (dissipation then occurs purely
through shocks).  From a more quantitative perspective, the scaling
exponents of the solutions in the small viscosity limit have attracted
considerable interest, both in the physics and the applied mathematics
literature
\cite{PhysRevLett.77.3118,PhysRevE.54.5116,Kraich,EVan00,EVan00Stat}.

The white noise term $\xi$ in~\eref{e:BurgersBasic} leads to solutions
$u$ which, in general, will be very ``rough''.  In particular, $u$ as
a function of the spatial variable~$x$ will not be differentiable in
the classical sense, but will only possess some H\"older regularity.
As a consequence, it transpires that the solutions to
\eref{e:BurgersBasic} are extremely unstable under natural
approximations of the nonlinearity, and this is the phenomenon that
will be explored in this article.  For example, for any $a,b \ge 0$
with $a+b > 0$, one can consider the approximating equation
\begin{equ}[e:approx]
  \d_t u^\eps(x,t)
    = \nu\,\d_x^2 u^\eps
        - g(u^\eps)\, D_\eps u^\eps(x,t)
        + \sigma\,\xi(x,t),
\end{equ}
where the approximate derivative $D_\eps$ is defined as
\begin{equ}[e:defDe]
  D_\eps u(x,t) = {u(x + a\eps,t) - u(x-b\eps,t) \over (a+b)\eps}.
\end{equ}
In the absence of the noise term $\xi$ it would be a standard exercise
in numerical analysis to show that the solution of \eref{e:approx}
converges to the solution of \eref{e:BurgersBasic} as $\eps \to
0$. This is just an example of the widely accepted `folklore' fact
that, if an equation is well-posed, any `reasonable' approximation
will converge to the exact solution.\footnote{Remember that we are
  working at fixed non-zero viscosity here, so there is no ambiguity in the
  concept of solution and we do not require an
  upwind scheme in order to obtain convergence.}

In this article, we will argue that, if $\xi$ is taken to be space-time
white noise, the limit of \eref{e:approx} as $\eps \to 0$ depends on
the values $a$ and $b$ and is equal to \eref{e:BurgersBasic}
only if either $g$ is constant or $a = b$!  Furthermore, it will
follow from the argument that, if one considers driving noise that is
slightly rougher than space-time white noise (taking a noise term
equal to $(1-\d_x^2)^{\alpha}\,dw(t)$ with $\alpha \in(0,1/4)$ still
yields a well-posed equation), one does not expect solutions to
the approximate equation \eref{e:approx} to converge to anything at
all, unless $a =b$.  Our methodology here is to first present a
heuristic argument which allows to derive quantitative predictions for
the effect of the finite difference discretisation on the solution.
We will then use numerical experiments to verify these predictions.

At this point we would like to emphasise that the aim of this article
is certainly not to advocate the use of a finite difference scheme of
the type \eref{e:approx} to effectively simulate
\eref{e:BurgersBasic}. Indeed, we will show in
section~\ref{sec:regular} below that approximations of the
nonlinearity of the type $D_\eps G(u)$, where $G$ is the
antiderivative of $g$ already have much better stability properties.
Instead, our aim is merely to give a striking illustration of the fact
that caution should be exercised in the simulation of stochastic PDEs
driven by spatially rough noise.

The text is structured as follows: We start in
section~\ref{sec:correction} by presenting our argument for the case
of the stochastic Burgers equation, \textit{i.e.}\ for $g(u)=u$.  In
section~\ref{S:general} we will present the corresponding results for
more general equations and in section~\ref{S:viscosity limit} we study
the limit of vanishing noise and viscosity.  Finally, in
appendix~\ref{S:simulate} we discuss some technical aspects
of the simulations used throughout this article.

\subsection*{Acknowledgements}

We would like to thank Andrew Majda, Andrew Stuart, Eric Vanden
Eijnden, and Sam Falle for helpful discussions of the phenomena
discussed in this article.  Financial support was kindly provided by
the EPSRC through grants EP/E002269/1 and EP/D071593/1, as well as by
the Royal Society through a Wolfson Research Merit Award.  JV would
like to thank the Courant Institute, where work on this article was
started, for its hospitality.


\section{Stochastic Burgers Equation}
\label{sec:correction}

In this section we consider the stochastic Burgers equation
\begin{equ}[e:Burgers]
  du = \nu\,\d_x^2 u\,dt - u\,\d_x u\,dt + \sigma\,dw,
\end{equ}
as well as the approximation given by
\begin{equ}[e:BurgersApprox]
  du^\eps = \nu\,\d_x^2 u^\eps\,dt - u^\eps\,D_\eps u^\eps\,dt + \sigma\,dw.
\end{equ}
The solutions $u$ and $u^\eps$ take values in the space
$L^2\bigl([0,2\pi],\R\bigr)$ on which the operator $\d_x^2$ is
endowed with periodic boundary conditions, and $w$ is an
$L^2$-cylindrical Wiener process (see \cite{ZDP} for details).
Equation~\eref{e:Burgers} is well-posed, since we can rewrite $u\,\d_x
u$ as ${1\over 2}\d_x(u^2)$ which is locally Lipschitz from the
Sobolev space $H^{1/4}$ into $H^{-1}$, thus allowing us to apply
general local well-posedness theorems as in \cite{ZDP,SPDENotes}.  For
fixed time $t>0$, the solutions to \eref{e:Burgers} have the
regularity of Brownian motion when viewed as a function of the
spatial variable $x$.  In particular, they are not differentiable
in~$x$.  Figure~\ref{fig:solutions} shows numerical solutions of
equation~\eref{e:BurgersApprox} for different values of $a$ and~$b$.
One can see that different choices for these parameters lead to an
$\CO(1)$ difference in the solutions.

Our aim in this article is to understand and quantify these
differences. In particular, in conjecture~\ref{C1} below, we compute a
correction term to \eref{e:Burgers} and we verify numerically that the
solutions to \eref{e:BurgersApprox} converge to the corrected
equation as $\eps \to 0$.  This understanding will then allow us to
conjecture the appearance of similar correction terms in more
complicated situations and we will again verify these conjectures
numerically.

\subsection{Heuristic Explanation}

\begin{figure}
  \begin{center}
    \includegraphics{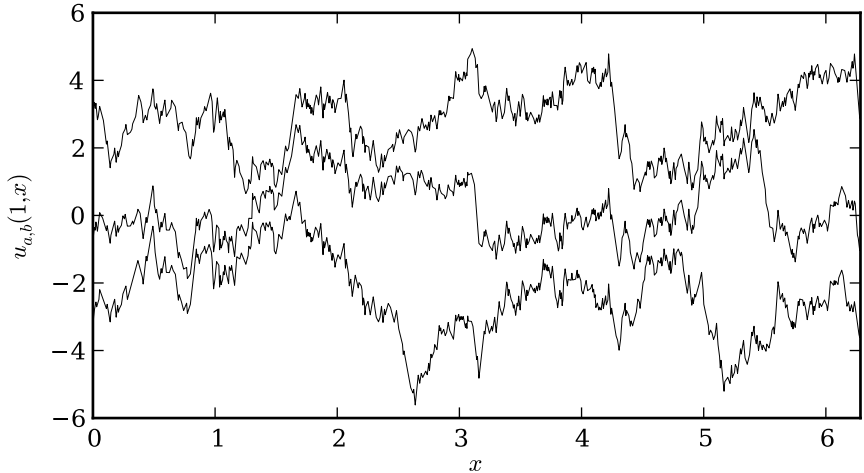}
  \end{center}
  \caption[discretisation error for the FD
  method]{\label{fig:solutions}\it Illustration of the discretisation
    error for the finite difference method~\eref{e:BurgersApprox}.
    The three lines correspond to a right-sided discretisation ($a=1$,
    $b=0$; top-most curve), a centred discretisation ($a=1$, $b=1$; middle
    curve) and a left-sided discretisation ($a=0$, $b=1$; bottom-most
    curve), all computed using the same instance of the driving noise.  The
    picture clearly shows that there is an $\CO(1)$ difference between
    the solutions obtained by the three different discretisation
    schemes.  From the argument presented in the text, we assume that
    the exact solution of~\eref{e:Burgers} will be closest to the
    middle of the three lines.}
\end{figure}

For simplicity, instead of $u$ we consider the solution $v$ to the
stochastic heat equation
\begin{equ}[e:heat]
  dv = \nu\,\d_x^2 v\,dt  + \sigma\,dw(t).
\end{equ}
Since the properties of the discretisation of differential operators
only depend on local properties, and since $v$ has the same spatial
regularity as $u$, it will be sufficient to study how well $v\, D_\eps
v$ approximates $v\,\d_x v = {1\over 2} \d_x v^2$.

By expressing $v$ in the Fourier basis
$\bigl\{\e^{inx}/\sqrt{2\pi}\bigr\}_{n\in\Z}$ it is easy to check
that the stationary solution to \eref{e:heat} is
\begin{equ}
  v(t,x)
  = \sum_{n \in \Z\setminus\{0\}}
    \frac{\sigma}{in\sqrt{2\nu}} \xi_n(t) \frac{\e^{inx}}{\sqrt{2\pi}}
    + \xi_0(t) \frac{1}{\sqrt{2\pi}},
\end{equ}
where the $\xi_0$ is a (real-valued) standard Brownian motion and
$\xi_n$ for $n\neq 0$ are complex-valued Ornstein-Uhlenbeck processes
with variance $1$ (in the sense that $\E |\xi_n(t)|^2 = 1$) and time
constant $\nu n^2$ which are independent, except for the condition
that $\xi_{-n} = \bar \xi_n$.  Therefore, the derivative of $v$ is
given (at least formally) by
\begin{equ}[e:exactv]
  \d_x v(x) = \sum_{n\neq 0} \frac{\sigma \xi_n(t) \e^{inx}}{2\sqrt{\nu\pi}}.
\end{equ}
The $\eps$-approximation to the derivative given in \eref{e:defDe},
on the other hand, is given by
\begin{equ}[e:approxv]
  D_\eps v(x)
  = \sum_{n \neq 0} {\sigma \xi_n(t) \e^{inx} \over 2\sqrt{\nu\pi}}
        {\e^{ina\eps} - \e^{-inb\eps}\over (a+b)i\eps n}.
\end{equ}
It is clear that the terms in \eref{e:approxv} are a good
approximation to the terms in \eref{e:exactv} only up to $n \approx
\eps^{-1}$.  For larger $n$, the multiplier in \eref{e:approxv} will
decrease like $n^{-1}$.

For our analysis we restrict ourselves to the constant ($n=0$) Fourier
mode.  Our numerical experiments, below, show that the contributions
from this mode are already enough to explain the observed differences
between the solutions of the approximating
equation~\eref{e:BurgersApprox} and the exact solution.  Since
$v\,\d_x v$ is a total derivative, the $0$-mode of this term vanishes.
In contrast, the $0$-mode of $v\, D_\eps v$ does not vanish at all: We
obtain instead for this term the sum
\begin{equs}[e:contrib]
  \bigl\la v D_\eps v\bigm| \frac{1}{\sqrt{2\pi}} \bigr\ra
  &= \frac{1}{\sqrt{2\pi}} \sum_{n\neq 0} \frac{\sigma^2 \xi_{-n}(t) \xi_n(t)}{2\nu (-in)}
       {\e^{ina\eps} - \e^{-inb\eps}\over (a+b)\eps in} \\
  &= \frac{\sigma^2}{\sqrt{2\pi}\nu} \sum_{n>0} |\xi_n(t)|^2
        {\cos a\eps n - \cos b\eps n \over (a+b)\eps n^2}
\end{equs}
and the expectation of this expression, as $\eps \to 0$, converges
to
\begin{equs}
  \lim_{\eps\downarrow0} \E\bigl( -v D_\eps v\bigr)(x)
   &= \lim_{\eps\downarrow0}
      \E\bigl\la -v D_\eps v\bigm| \frac{1}{\sqrt{2\pi}} \bigr\ra \frac{1}{\sqrt{2\pi}} \\
  &= - {\sigma^2\over 2\pi\nu} \int_0^\infty {\cos a\kappa - \cos b\kappa \over (a+b)\kappa^2}\,d\kappa
  = {\sigma^2 \over 4\nu} {a-b \over a+b}.
\end{equs}
As a consequence, one expects the following result.

\begin{conjecture}\label{C1}
  The solution of the approximating equation~\eref{e:BurgersApprox}
  converges, as $\eps \to 0$, to the solution of
  \begin{equ}[e:BurgersWrong]
    du = \nu\,\d_x^2 u\,dt - u\,\d_x u\,dt
           + \frac{\sigma^2}{4\nu}\frac{a-b}{a+b}\,dt + \sigma\, dw.
  \end{equ}
  Thus, the approximation converges to the stochastic Burgers
  equation~\eref{e:Burgers} only for $a = b$.
\end{conjecture}

\begin{remark}
  The solution to the stochastic Burgers equation (or rather the
  integrated process which solves the corresponding KPZ equation)
  arising as the fluctuations process in the weakly asymmetric
  exclusion process \cite{BertGia97,BQS09} is driven by the
  \textit{derivative} of space-time white noise. As a consequence, it
  does \textit{not} solve an SPDE that is well-posed in the classical
  sense and can currently only be defined via the Hopf-Cole transform.
  Such a process is even rougher (by ``one derivative'') than the
  process considered here and one would expect the `wrong' numerical
  approximation schemes to fail there in an even more spectacular way.
\end{remark}

\begin{remark}
  One may think of two reasons why the correction term vanishes when
  $a=b$. On the one hand, this discretisation is more symmetric. On
  the other hand, it yields a second-order approximation to the
  derivative at $x$. The correct explanation is closer to the first
  one. Indeed, consider for example the discretisation
  \begin{equ}[e:nonsym]
    \bigl( \tilde D_\eps u\bigr)(x)
    \approx \frac{c\, u(x+2\eps) + (1-3c)\, u(x+\eps) + 3c\, u(x) - (1+c)\, u(x-\eps)}{2\eps}.
  \end{equ}
  This discretisation is of second order for every $c\in\R$. However,
  if we perform the same calculation as above with this
  discretisation, we obtain a correction term equal to
  \begin{equ}
    {\sigma^2\over 2 \pi\nu} \int_0^\infty {c\cos 2\kappa -4c\cos \kappa + 3c \over 2\kappa^2}\,d\kappa
    =  - {c\sigma^2 \over 8\nu},
  \end{equ}
  which vanishes only if $c=0$, \textit{i.e.}\ when \eref{e:nonsym} conincides
  with the symmetric discretisation~\eref{e:defDe}.
\end{remark}

In the above calculation, both the limiting equation and the
approximating equation live in the same space.  It is possible to
carry out a similar analysis in the case where the approximating
equation takes values in a different space, typically $\R^N$ for some
large $N$.  For example, we can consider the `finite difference'
approximation given by
\begin{equs}[e:burgersFD]
  \d_x^2 u &\approx \frac{u(x+\delta) - 2 u(x) + u(x-\delta)}{\delta^2}
            \eqdef \bigl(\Delta_N u\bigr)(x) \\
  u\,\d_x u &\approx u\, D_\delta u
            =  u(x) \frac{u(x + \delta) - u(x)}{\delta}
            \eqdef F_N (u)(x),
\end{equs}
where we set $\delta = {2\pi \over N}$ for the approximation with $N$
gridpoints. In this setting, we approximate $u$ by $u^N \in \R^N$ with
$u^N_j \approx u(j\delta)$.  The natural candidate for the
approximation of space-time white noise is then given by $dW^N_j =
\delta^{-1/2}\, dW_j$, where the $W_j$ are independent,
one-dimensional standard Brownian motions.  This is the correct
scaling, since it ensures for example that for every continuous
function $g$, $\sum_j W^N_j g(\delta j)$ is a Wiener process with
covariance $\sum_j g(\delta j)^2 \delta \approx \int g^2\,dx$.  With
this notation, we consider the approximating equation given by
\begin{equ}[e:approxN]
  du^N = \nu \Delta_N u^N \,dt - F_N(u^N)\,dt + \sigma\,dW^N(t).
\end{equ}

Let us take $N$ even for the sake of simplicity.  It is then
straightforward to check that the eigenvectors for $\Delta_N$ are
given as before by $\e^{inx}$ with $n=-{N\over 2}+1,\ldots,{N\over
  2}$, but the corresponding eigenvalues are
\begin{equ}
  \lambda_n
  = {2\over \delta^2} \bigl(\cos n\delta - 1\bigr)
  = - \Bigl({2\over \delta} \sin  \bigl({n \delta\over 2}\bigr)\Bigr)^2 \eqdef - \eta_n^2.
\end{equ}
(Note that for fixed $n$ and small $\delta$, one has indeed $\lambda_n
\approx - n^2$.) It then follows as previously that the solution to
the linearised equation is given by
\begin{equ}
  v(t,x)
  = \sum_{n \neq 0}
    \frac{\sigma}{2\sqrt{\nu \pi} i \eta_n} \e^{inx} \xi_n(t)
    + \xi_0(t) \frac{1}{\sqrt{2\pi}}
\end{equ}
and that its discrete derivative $D_\delta v$ is given by
\begin{equ}
  D_\delta v(t, x)
  = \sum_{n \neq 0} {\sigma \e^{inx} \xi_n(t) \over 2\sqrt{\nu \pi} i \eta_n} {\e^{in\delta} - 1\over \delta}.
\end{equ}
Note that the sums in both expressions only run over the
values $n = -{N\over 2}+1,\ldots,{N\over 2}$.
Similarly to above, we obtain that the expectation of the zero mode of
the product $v\,D_\delta v$ is given by
\begin{equ}[e:corr1/4]
  \E \bigl\la -v D_\eps v\bigm| \frac{1}{\sqrt{2\pi}} \bigr\ra \frac{1}{\sqrt{2\pi}}
  = - \frac{\sigma^2}{2\pi\nu} \sum_{n=1}^{N/2} \frac{\cos \delta n - 1}{\delta \eta_n^2}
  = \frac{\sigma^2}{2\pi\nu} \sum_{n=1}^{N/2} \frac{\delta}{2}
  = {\sigma^2 \over 4\nu}.
\end{equ}
One therefore expects the following result.

\begin{conjecture}\label{C2}
  The solution $u^N$ of the finite difference approximation~\eref{e:approxN} converges, as
  $N\to\infty$, to the solution of
  \begin{equ}
    du = \nu\,\d_x^2 u\,dt - u \,\d_x u\,dt + {\sigma^2 \over 4\nu}\,dt + \sigma\,dw(t).
  \end{equ}
\end{conjecture}

\begin{figure}
  \begin{center}
    \includegraphics{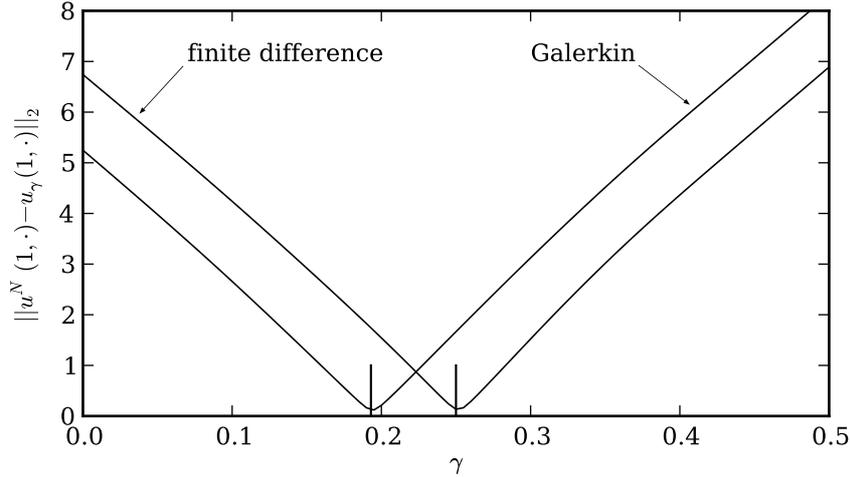}
  \end{center}
  \caption[]{\label{fig:burgers}\it This figure compares the solution
    $u^N$ of the finite difference approximation~\eref{e:approxN}
    to the solution $u_\gamma$ of the ``corrected''
    SPDE~\eref{e:corrected} (which includes an additional drift
    term~$\gamma\sigma^2/\nu$).  The curves show the $L^2$-norm
    difference between the solutions (for the same instance of the
    noise) as a function of $\gamma$, once using finite difference
    discretisation~\eref{e:burgersFD} of the linear part
    (corresponding to the top-most curve in
    figure~\ref{fig:solutions}) and once using the spectral Galerkin
    discretisation.  The two vertical line segments give the predicted
    locations for the minima of the two curves.  It can be seen that
    predictions and simulations are in good agreement.}
\end{figure}

To test this conjecture, we use the following numerical experiment:
We numerically solve both the ``approximating'' equation~\eref{e:approxN}
and the ``corrected'' SPDE
\begin{equ}[e:corrected]
  du_\gamma
  = \nu\,\d_x^2 u_\gamma\,dt - u_\gamma\,\d_x u_\gamma\,dt
    + \gamma\frac{\sigma^2}{\nu}\,dt
    + \sigma\,dw(t),
\end{equ}
until a fixed time~$T$, using the same instance of the noise.  To
discretise the term $u\,\d_x u$ in \eref{e:corrected} we use the
approximation $\bigl(u(x)^2/2-u(x-\delta)^2/2\bigr)/\delta$, which is
known to converge to the exact solution, see section~\ref{sec:regular}
below.  Also, for increased accuracy, we use a finer grid
for~\eref{e:corrected} than we did for~\eref{e:approxN}.  Now we can
compare the two solutions by considering $\|u^N(T,\quark) -
u_\gamma(T,\quark)\|_2$ as a function of~$\gamma$: If
conjecture~\ref{C2} is correct, this function will take its minimum at
$\gamma \approx 1/4$.  The result of such a simulation is given by the
line labelled ``finite difference'' in figure~\ref{fig:burgers}.  The
minimum of the curve is indeed located close to~$\gamma=1/4$, thus
giving support to conjecture~\ref{C2}.

Even though the correction terms in conjectures \ref{C1} and~\ref{C2}
(with $a=1$ and $b=0$) coincide, the constants arise in completely
different ways.  This might lead to the speculation that the value of
this constant is an intrinsic property of the limiting equation,
rather than of the particular way of approximating it.  The following
argument shows that this is not the case.  One can repeat the
calculation leading to conjecture~\ref{C2} with a `spectral Galerkin'
approximation of the linear part of the equation, but retaining a
`finite difference' approximation of the nonlinearity.  In other
words, we consider \eref{e:approxN} as before, but we take for
$\Delta_N$ the self-adjoint matrix with eigenvectors $\e^{inx}$ and
eigenvalues $-n^2$. (This can be achieved by first applying the
discrete Fourier transform, then multiplying the $n$th component by
$-n^2$, and finally applying the inverse Fourier transform.)  In this
case one has $\eta_n = n$ and it transpires that the correction term
is given by
\begin{equ}
  \sum_{n =1}^{N/2} {\sigma^2 \over 2\pi\nu \eta_n^2} {1-\cos n\delta \over \delta}
  \approx {\sigma^2\over 2\pi\nu} \int_0^{\pi} {1-\cos \kappa\over \kappa^2}\,d\kappa
  \approx {0.193\,\sigma^2\over \nu}.
\end{equ}
which is clearly different from \eref{e:corr1/4}.

To verify that the spectral Galerkin discretisation of the linear part
indeed leads to this different correction term, we modify the code which
we used to validate conjecture~\ref{C2} above.  The result of this
simulation is given by the line labelled ``Galerkin'' in
figure~\ref{fig:burgers}.  Again, there is good agreement between our
conjecture and the simulation results.

\subsection{The Case of More Regular Noise}
\label{sec:regular}

To conclude this section, let us argue that the situation considered
in this article is truly a borderline case in terms of regularity and
that if we drive \eref{e:BurgersApprox} by noise that gives rise to
slightly more regular solutions, the solutions of~\eref{e:BurgersApprox}
converge to those of \eref{e:Burgers} without any correction
term. Indeed, consider a general semilinear stochastic PDE driven by
additive noise:
\begin{equ}[e:general]
  du = -Au\,dt + F(u)\,dt + Q\,dw(t),
\end{equ}
where $A$ is a strictly positive-definite selfadjoint operator on some
Hilbert space $\CH$, $F$ is a (possibly unbounded) nonlinear operator
from $\CH$ to $\CH$, $W$ is a standard cylindrical Wiener process on
$\CH$, and $Q\colon \CH \to \CH$ is a bounded operator.

Denote furthermore $\CH^\alpha = \CD(A^\alpha)$ for $\alpha > 0$ and
let $\CH^{-\alpha}$ be the dual space to $\CH^\alpha$ under the dual
pairing given by the Hilbert space structure of $\CH$. (So that
$\CH^{-\alpha}$ is a superspace of $\CH$ for $\alpha > 0$.) We also denote by $\|\cdot\|_\alpha$
the natural norm of $\CH^\alpha$. Finally,
we denote as before by $v$ the solution to the linearised system
\begin{equ}
  dv = -Av\,dt + Q\,dw(t),
\end{equ}
which we assume to be an $\CH$-valued Gaussian process with almost
surely continuous sample paths.  One then has the following
convergence result:

\begin{theorem}\label{theo:approx}
  Assume that there exists $\gamma \ge 0$ and $a \in [0,1)$ such that
  $F \colon \CH^\gamma \to \CH^{\gamma - a}$ is locally Lipschitz
  continuous and such that the process $v$ has continuous sample paths
  with values in $\CH^\gamma$.  Assume furthermore that $F_\eps\colon
  \CH^\gamma \to \CH^{\gamma - a}$ is such that $F_\eps$ is locally
  Lipschitz and such that the convergence $F_\eps(u) \to F(u)$ takes
  place in $\CH^{\gamma - a}$, locally uniformly in $\CH^\gamma$.
  Then, the solutions $u_\eps$ to
  \begin{equ}[e:generaleps]
    du_\eps = -Au_\eps\,dt + F(u_\eps)\,dt + Q\,dw(t),
  \end{equ}
  converge to the solutions to \eref{e:general} as $\eps \to 0$.
\end{theorem}

\begin{proof}
  The proof is straightforward and we only sketch it. We assume
  without loss of generality that the parameter $\eps$ is chosen in
  such a way that for every $R>0$ there exists a constant $C_R$ such
  that
  \begin{equ}[e:approxF]
    \sup_{\|x\|_\gamma \le R} \|F_\eps(u) - F(u)\|_{\gamma-a} \le C_R \eps.
  \end{equ}
  Denote now by $u$ the solution to \eref{e:general} with initial
  condition $u_0$ and by $u_\eps$ the solution to \eref{e:generaleps}
  with initial condition $u_0^\eps$. Let $t>0$ be small enough so that
  $\|u(s)\|_\gamma \le R$ and $\|u(s) - u_\eps(s)\|_\gamma \le R$ for
  $s \le t$.  We then have
  \begin{equs}
    \|u(t) - u_\eps(t)\|_\gamma
      &\le \|u_0 - u_0^\eps\|_\gamma + C\int_0^t (t-s)^{-a} \|F_\eps(u_\eps(s)) - F(u(s))\|_{\gamma - a}\,ds \\
      &\le \|u_0 - u_0^\eps\|_\gamma + C\int_0^t (t-s)^{-a} \|F_\eps(u_\eps(s)) - F(u_\eps(s))\|_{\gamma - a}\,ds \\
        &\qquad + C\int_0^t (t-s)^{-a} \|F(u_\eps(s)) - F(u(s))\|_{\gamma - a}\,ds \\
      &\le  \|u_0 - u_0^\eps\|_\gamma + C_R \eps + C_R t^{1-a} \sup_{s \le t} \|u(s) - u_\eps(s)\|_\gamma.
  \end{equs}
  The claim then follows by taking $t$ sufficiently small and
  performing a simple iteration.
\end{proof}

Despite its simplicity, this criterion is surprisingly sharp. Indeed,
we argue that if we consider \eref{e:Burgers} but with the space-time
white noise $dw$ replaced by $(1-\d_x^2)^{-\delta}\, dw$ for $\delta >
0$, then the assumptions of theorem~\ref{theo:approx} can be satisfied
with some choice of exponent $\gamma$ for the approximation
\begin{equ}
  F_\eps(u)(x) = u(x) {u(x+\eps) - u(x) \over \eps} \eqdef u(x) \bigl(D_\eps u\bigr)(x).
\end{equ}
Obviously, this cannot be the case when $\delta = 0$, since we then
observe the convergence to solutions to \eref{e:BurgersWrong}.

Indeed, we first note that since the linear operator appearing in
\eref{e:Burgers} is of \textit{second} order, we have the
correspondence
\begin{equ}
  \CH^{\gamma} = H^{2\gamma},
\end{equ}
between interpolation spaces and fractional Sobolev spaces. In order
to keep our notation coherent throughout this section, we still denote
by $\|\cdot\|_\gamma$ the norm of $\CH^\gamma$, \textit{i.e.}\
$\|u\|_\gamma = \|(1-\d_x^2)^\gamma u\|_{L^2}$, where we implicitly
endow $\d_x^2$ with periodic boundary conditions.  With this notation,
we get the following result.

\begin{lemma}\label{lem:approxDer}
  For $\gamma \ge 0$ and $a \in [{1\over 2},1]$, we have $\|D_\eps u -
  \d_x u\|_{\gamma-a} \le C \eps^{2a-1} \|u\|_\gamma$.
\end{lemma}

\begin{proof}
  The operator $D_\eps - \d_x$ is given by the Fourier multiplier
  $M_\eps(k) = \eps^{-1}\bigl(e^{ik\eps} - 1 - ik\eps\bigr)$.  We
  immediately obtain the bound
  \begin{equ}
    |M_\eps(k)| \le k\bigl(\eps k \wedge 1\bigr) \le k (\eps k)^{2a-1},
  \end{equ}
  from which the claim follows at once.
\end{proof}

Since furthermore $\CH^\gamma$ is an algebra for $\gamma > {1\over
  4}$, we conclude that the bound \eref{e:approxF} holds (for some
different power of $\eps$), provided that we choose $\gamma \in
({1\over 4},{1\over 2}]$ and $a = 2\gamma$.  On the other hand, the
solution to the linearised equation
\begin{equ}
  dv = \d_x^2 v\,dt + (1-\d_x^2)^{-\delta}\, dw
\end{equ}
belongs to $\CH^\gamma$ if and only if $\gamma < {1\over 4} + \delta$,
thus supporting our claim that the case $\delta = 0$ is precisely
borderline for the applicability of theorem~\ref{theo:approx}.

If on the other hand we make the more natural choice
\begin{equ}[e:goodF]
  F_\eps(u) = {1\over 2} D_\eps \bigl(u^2\bigr),
\end{equ}
then it turns out that we can apply theorem~\ref{theo:approx} even in
the case $\delta = 0$. Indeed, it follows from standard Sobolev theory
(see for example \cite{SPDENotes}) that if $\gamma \in ({1\over
  8},{1\over 4})$, the map $u \mapsto u^2$ is locally Lipschitz from
$\CH^\gamma$ into $\CH^\beta$ provided that $\beta < 2\gamma - {1\over
  4}$.  As a consequence, it follows from lemma~\ref{lem:approxDer}
that the approximation $F_\eps$ given by \eref{e:goodF} converges to
$u \,\d_x u$ in the sense of \eref{e:approxF}, provided that $\gamma \in
({1\over 8},{1\over 4})$ and $a \in ({3\over 4}-\gamma,1)$.

\begin{remark}
  Some \textit{ad hoc} numerical scheme was shown to converge to the exact
  solution in \cite{AlGy06}. Some other schemes are shown to converge
  in \cite{GuKieNi02}, but only in the case $\delta > 0$ of more
  regular noise.  A particle approximation to a specific modification
  was constructed in \cite{GugDuan}.
\end{remark}

\subsection{Numerical Verification of the Borderline Case}

We have performed numerical simulations that corroborate the argument
presented in the previous section and show that $\delta = 0$ truly is
the borderline case for the limiting equation to be independent on the
type of discretisation performed on the nonlinearity.  In the case
$\delta < 0$ (\textit{i.e.}\ the case where the driving noise is
\textit{rougher} than space-time white noise), our preceding
discussion suggests that the centred discretisation should converge to
the correct solution for $|\delta|$ sufficiently small, but that the
solutions to both the left-sided and the right-sided discretisations
should diverge as $\eps \to 0$.

In order to verify this effect, we numerically solve the SPDE
\begin{equ}[e:roughBurgers]
  du = \nu\,\d_x^2 u\,dt - u\,\d_x u\,dt + (1-\d_x^2)^{-\delta} \,dw,
\end{equ}
using different discretisation schemes and different values for
$\delta$.  The results are shown in figure~\ref{fig:rough}.  The
simulations are in good agreement with the argument outlined above.

\begin{figure}
  \begin{center}
    \includegraphics{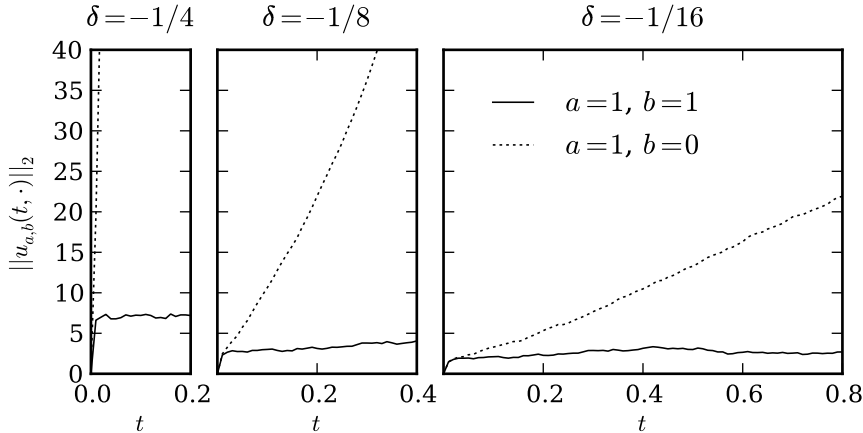}
  \end{center}
  \caption{\label{fig:rough}\it Illustration of the divergence of the
    right-sided discretisation for noise rougher than space-time white
    noise.  The three panels show, for different values of $\delta$,
    the $L^2$-norm of numerical solutions to
    equation~\eref{e:roughBurgers}.  The full lines correspond to a
    centred discretisation ($a=1$, $b=1$) whereas the dotted lines
    correspond to the right-sided discretisation ($a=1$, $b=0$).  The
    figures show that only the centred discretisation seems to be
    stable for $\delta<0$.}
\end{figure}


\section{Possible Generalisations of the Argument}
\label{S:general}

In this section, we discuss a number of possible extensions of these
results to more general Burgers-type equations.  We restrict ourselves
in this discussion to the case $a=1$, $b=0$, \textit{i.e.}\ to
right-sided discretisations. This is purely for notational
convenience, and one would expect similar correction terms to appear
for arbitrary values of $a$ and $b$, just as before.

\subsection{More General Nonlinearities}

Consider the equation
\begin{equ}[e:BurgersGen]
  du_i = \nu\,\d_x^2 u_i\,dt + \sum_{j=1}^d\d_j h_i(u) \d_x u_j\,dt + \sigma\,dw_i,
\end{equ}
for an $\R^d$-valued process $u$ and a smooth function $h \colon \R^d
\to \R^d$ with bounded second and third derivatives.  Rewriting the
nonlinearity as $\d_x \bigl(h_i(u)\bigr)$, we see that this equation
is globally well-posed.  As before, we consider the approximating
equation
\begin{equ}[e:approxGen]
  du_i^\eps(x,t)
  = \nu\,\d_x^2 u_i^\eps(x,t)\,dt +\sum_{j=1}^d \d_j h_i\bigl(u^\eps(x,t)\bigr) D_\eps u_j^\eps(x,t)\,dt + \sigma\,dw_i(t).
\end{equ}
The idea now is to introduce a cut-off frequency $N$ and to write
$u^\eps = \bar u^\eps + \tilde u^\eps$, where $\bar u^\eps$ is the
projection of $u^\eps$ onto Fourier modes with $|k| \le N$. Since the
linear part of the equation dominates the nonlinearity at high
frequencies, one expects $\tilde u^\eps$ to be well approximated by
$\tilde v$, the projection of $v$ onto the high frequencies.  This on
the other hand is small in the $L^\infty$ norm (it decreases like
$N^{-s}$ for every $s < {1\over 2}$), so that
\begin{equ}
  \d_j h_i (u^\eps)
  \approx \d_j h_i (\bar u^\eps) + \sum_{k=1}^d \d^2_{jk} h_i(\bar u^\eps) \tilde v_k.
\end{equ}
It now follows from the same argument as before that the term
$\d^2_{jk} h_i(\bar u^\eps) \tilde v_k D_\eps v_j$ is expected to
yield a non-vanishing contribution for $k = j$ in the limit $\eps \to
0$ and $N\to\infty$.  Provided that we keep $N \ll {1\over \eps}$,
this contribution will again be described by \eref{e:contrib}, so that
we expect the following behaviour:

\begin{conjecture}\label{C3}
  The solution of \eref{e:approxGen} converges, as $\eps \to 0$, to
  the solution of
  \begin{equ}[e:general,corrected]
    du_i = \nu\,\d_x^2 u_i\,dt
             +  \sum_j\Bigl( \d_j h_i(u) \d_x u_j - {\sigma^2\over 4\nu}\d^2_{jj} h_i(u)\Bigr)\,dt
             + \sigma\,dw_i.
  \end{equ}
\end{conjecture}

In the one-dimensional case we can recover conjecture~\ref{C1} (for
$a=1, b=0$) from conjecture~\ref{C3} by choosing~$h(u) = -u^2/2$
in~\eref{e:general,corrected}.

We perform the following numerical experiment to validate the
functional form of the correction term given in conjecture~\ref{C3}:
We numerically solve both the ``approximating''
equation~\eref{e:approxGen} and, for $p\colon \R\to \R$, the
``corrected'' SPDE
\begin{equ}[e:correctedGen]
  du_p = \nu\,\d_x^2 u_p\,dt
           + h'(u_p) \d_x u_p \,dt
                  - {\sigma^2\over 4\nu}p(u_p) \,dt
           + \sigma\,dw
\end{equ}
until a fixed time~$T$, using the same instance of the noise.
As for the discretisation of~\eref{e:corrected} above, we use the
approximation $\bigl(h(x)-h(x-\delta)\bigr)/\delta$ for the term
$h'(u) \,\d_x u$ in the proposed limit~\eref{e:correctedGen} and we
also solve~\eref{e:correctedGen} on a finer grid than we use
for~\eref{e:approxGen}.  Finally, we numerically optimise the
correction term~$p$ (using some parametric form) in order to minimise
the distance $\|u^N(T,\quark) - u_p(T,\quark)\|_2$.  If the conjecture
is correct, we expect the minimum to be attained for a function $p$
which is close to the predicted correction term $h''$.  The result of
a simulation is shown in figure~\ref{fig:general}.  The figure shows
that there is indeed a good fit between the conjectured and
numerically determined correction terms.

\begin{figure}
  \begin{center}
    \includegraphics{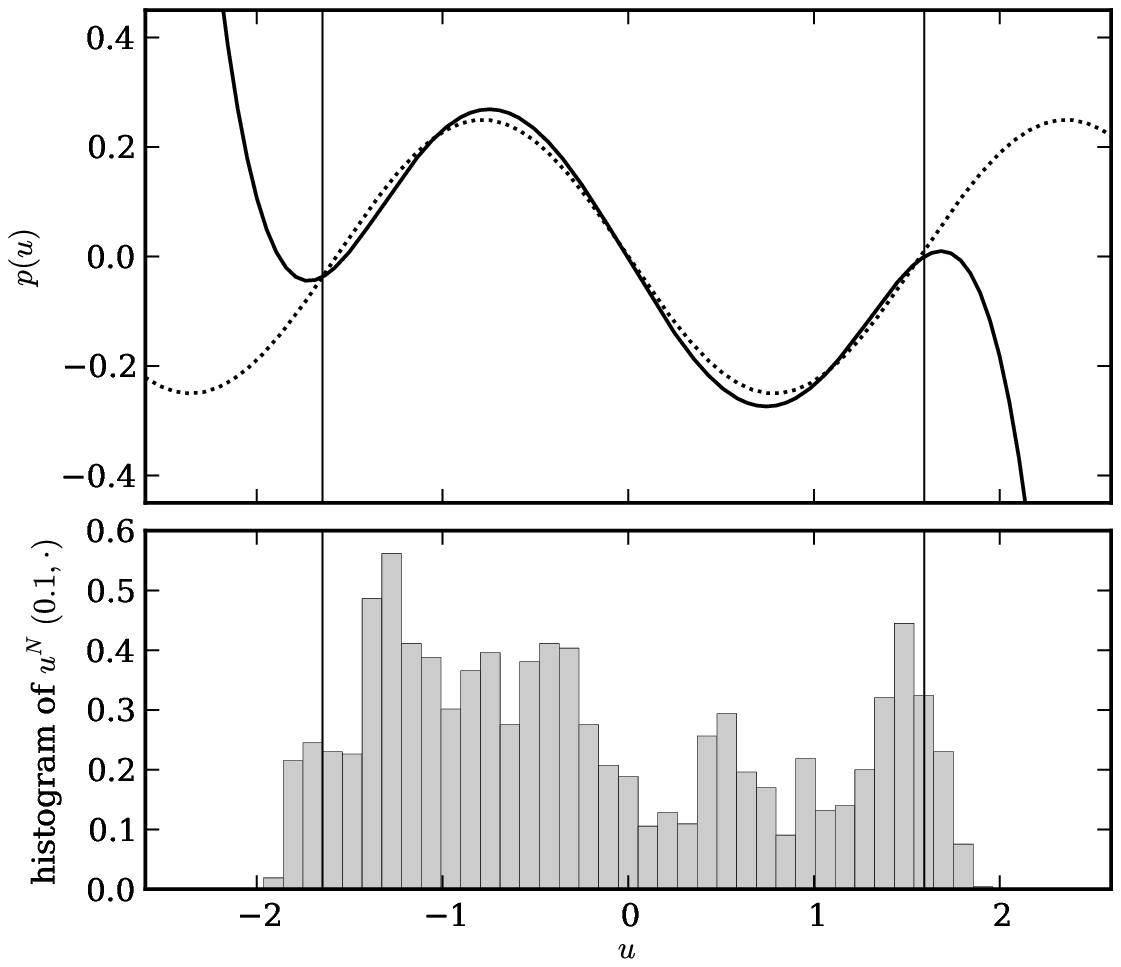}
  \end{center}
  \caption[discretisation error for more general
  equations]{\label{fig:general}\it Illustration of the convergence of
    \eref{e:approxGen} to~\eref{e:general,corrected} for the
    one-dimensional example $h'(x) = \sin(x)^2$.  For the figure we
    numerically compute the finite difference solution $u^N$ to
    \eref{e:approxGen}.  We then compute solutions $u_p$
    to~\eref{e:correctedGen}, using a fifth-order polynomial for the
    correction term~$p$.  This polynomial is then numerically fitted
    to minimise $\|u^N(1,\quark) - u_p(1,\quark)\|_2$.  The top panel
    shows the resulting fitted correction term $-\sigma^2 p/4\nu$
    (full line) together with the correction term $-\sigma^2
    h''(u)/4\nu$ predicted in conjecture~\ref{C3} (dotted line).  To
    give an idea which range of the correction term is actually used
    in the computation, the lower panel shows the histogram of the
    values of~$u^N$ (the vertical bars indicate the 5\% and 95\%
    quantiles).  Between these quantiles, the graph shows a good fit
    between the numerically determined and conjectured correction
    terms.}
\end{figure}

\subsection{Classically Ill-Posed Equations}

Pushing further the class of equations considered in the previous
subsection, one may want to consider equations of the form
\begin{equ}[e:nongrad]
  du_i
  = \nu\,\d_x^2 u_i\,dt
    + \sum_{j=1}^d G_{ij}(u) \d_x u_j\,dt
    + f(u)\,dt
    + \sigma\,dw_i
\end{equ}
for some functions $G\colon \R^d \to \R^{d\times d}$ and
$f\colon \R^d\to \R^d$.  If we do
\textit{not} assume that $G$ has an antiderivative and since solutions
are only expected to be $\alpha$-H\"older continuous in space for
$\alpha < {1\over 2}$, it is no longer even clear what it means to be
a solution to this equation. So, at least classically,
\eref{e:nongrad} is ill-posed and the mere \textit{concept} of
solutions to such an evolution equation is difficult to establish.

However, the discretised equation does of course still make sense for
any fixed value of $\eps$. Furthermore, we observed numerically that
there seems to be no instability as $\eps \to 0$; indeed one observes
pathwise convergence to a limiting process. By analogy with the
behaviour observed for the situations where \eref{e:nongrad} is
classically well-posed (\textit{i.e.}\ when $G$ has an
antiderivative), one would then be tempted to define solutions to
\eref{e:nongrad} to be those processes that can be obtained as limits
as $\eps \to 0$ of the solutions to the equation where $\d_x u_j$ is
replaced by its symmetric discretisation.  In analogy with
conjecture~\ref{C3}, we would then expect solutions of the right-sided
discretisation to converge, as $\eps \to 0$, to solutions of the
corrected equation
\begin{equ}
  du_i
  = \nu\,\d_x^2 u_i\,dt
    + \sum_j\bigl( G_{ij}(u) \d_x u_j - {\sigma^2\over 4\nu}\d_j G_{ij}(u) \bigr) \,dt
    + f(u)\,dt
    + \sigma\,dw_i.
\end{equ}
This reasoning leads to:

\begin{conjecture}\label{C4}
  As $\eps\to 0$, the equations
  \begin{equ}[e:C4.1]
    du^\eps_i
    = \nu\,\d_x^2 u^\eps_i\,dt
      + \sum_{j=1}^d G_{ij}(u^\eps) D_\eps^{1,0} u^\eps_j\,dt
      + f(u^\eps)\,dt
      + \sigma\,dw_i,
  \end{equ}
  where $D_\eps^{1,0}$ denotes the right-sided discretisation, and
  \begin{equ}[e:C4.2]
    d\tilde u^\eps_i
    = \nu\,\d_x^2 \tilde u^\eps_i\,dt
      + \sum_{j=1}^d \bigl( G_{ij}(\tilde u^\eps) D_\eps^{1,1} \tilde u^\eps_j - {\sigma^2\over 4\nu}\d_j G_{ij}(\tilde u^\eps) \bigr) \,dt
      + f(\tilde u^\eps)\,dt
      + \sigma\,dw_i,
  \end{equ}
  where $D_\eps^{1,1}$ denotes centred discretisation, converge to the
  same limit.
\end{conjecture}

To test conjecture~\ref{C4}, we use the following numerical
experiment: as an example we consider the SPDE
\begin{equation}\label{e:strangeSPDE}
\begin{split}
  \d_t u &= \frac{1}{\sigma^2} \,\d_x^2 u
             + \frac{2}{\sigma^2} \begin{pmatrix}
               0 & \cos(u_2)-\sin(u_1) \\
               \sin(u_1)-\cos(u_2) & 0
             \end{pmatrix} \d_x u \\
           & \hskip5cm
             + \frac{4}{\sigma^2} \begin{pmatrix}
               \phantom{+} \sin(u_1)\cos(u_1) \\
               - \cos(u_2)\sin(u_2)
             \end{pmatrix}
             + \sqrt{2}\, \partial_t w.
\end{split}
\end{equation}
This SPDE is of the form \eqref{e:nongrad} where $G$ has no
antiderivative.  SPDEs like \eref{e:strangeSPDE} occur in the problem
described in \cite[section 9]{HaiStuaVo07} where we argue that the
stationary distribution of this SPDE on $L^2\bigl([0,2\pi],
\R^2\bigr)$ (when equipped with appropriate boundary conditions)
coincides with the distribution of the stochastic differential
equation
\begin{equation*}
  dU(t) = 2 \begin{pmatrix} -\sin(U_2(t)) \\ \phantom{+}\cos(U_1(t)) \end{pmatrix} dt
       + \sigma \,dB(t).
\end{equation*}
For our experiment we numerically solve the SPDEs \eref{e:C4.1}
and~\eref{e:C4.2} and compare the solutions.  The result is displayed
in figure~\ref{fig:strange}.  As can be seen from the figure, the
simulation results are in good agreement with conjecture~\ref{C4}.

\begin{figure}
  \begin{center}
    \includegraphics{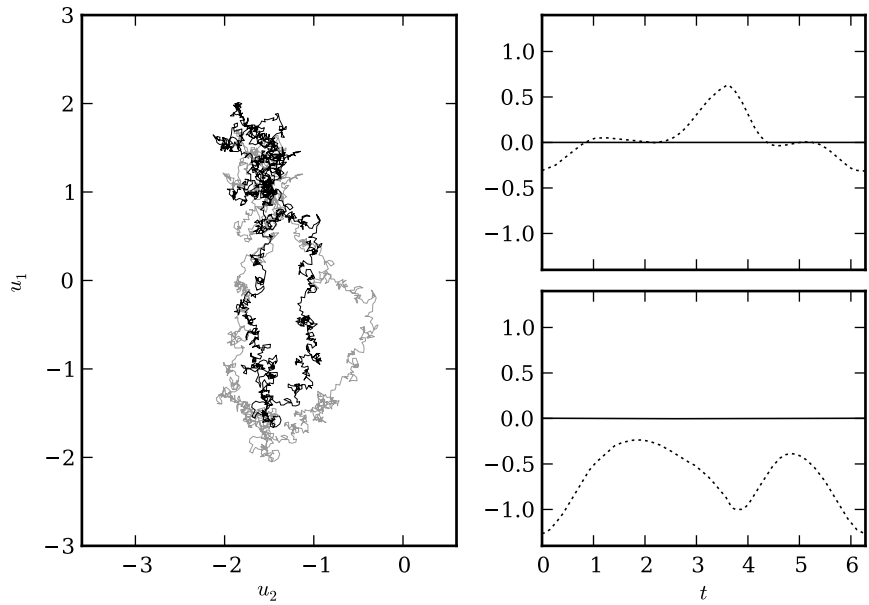}
  \end{center}
  \caption{\label{fig:strange}\it Illustration of the convergence in
    conjecture~\ref{C4}.  The left-hand panel shows a numerical
    solution of the $\R^2$-valued SPDE~\eref{e:strangeSPDE} at a fixed
    time~$t$.  Since we use periodic boundary conditions, the plotted
    graph of $u(t, \quark)$ forms a loop in~$\R^2$.  The black line in
    this plot was obtained using a centred discretisation, whereas the
    grey line in the background was obtained using a right-sided
    discretisation.  As in figure~\ref{fig:solutions}, there is an
    $O(1)$ difference between the two discretisation schemes.  The two
    plots on the right show the differences $u^\eps_1 - \tilde
    u^\eps_1$ (upper panel) and $u^\eps_2 - \tilde u^\eps_2$ (lower
    panel) between the discretisation schemes \eref{e:C4.1}
    and~\eref{e:C4.2} from conjecture~\ref{C4}.  (full lines).  For
    comparison, the plots also show the differences $u^\eps_1 - \bar
    u^\eps_1$ and $u^\eps_2 - \bar u^\eps_2$ where $\bar u^\eps$ is
    the solution of the uncorrected SPDE~\eref{e:strangeSPDE} using a
    centred discretisation (dotted lines).  The graphs support
    conjecture~\ref{C4} by showing good agreement between the
    solutions of \eref{e:C4.1} and~\eref{e:C4.2}.}
\end{figure}

\subsection{Multiplicative Noise}

We conclude this section by considering the equation
\begin{equ}[e:multBurgers]
  du = \nu\,\d_x^2 u\,dt + g(u) \,\d_x u\,dt + f(u)\,dw,
\end{equ}
where $g$ is as before and $f$ is a smooth bounded function with
bounded derivatives of all orders. Such an equation is well-posed if
the stochastic integral is interpreted in the It\^o sense
\cite{GyoNu}. (Note that it is \textit{not} well-posed if the
stochastic integral is interpreted in the Stratonovich sense.  This
follows from the fact that, at least formally, the It\^o correction
term is infinite when $f$ is not constant.)

In such a case, the local quadratic variation of the solution is
expected to be proportional to $f^2(u)$, so that one expects the
right-sided discretisation to exhibit a correction term proportional
to $g'(u) f^2(u)$. More precisely, in analogy to conjecture~\ref{C3},
one would expect the following statement to hold.

\begin{conjecture}\label{C5}
  The solution of
  \begin{equ}[e:C5.1]
    du = \nu\,\d_x^2 u\,dt + g(u) D_\eps u\,dt + f(u)\,dw
  \end{equ}
  converges, as $\eps \to 0$, to the solution of
  \begin{equ}[e:C5.2]
    du = \nu\,\d_x^2 u\,dt + g(u)\,\d_x u\,dt
           - \frac{1}{4\nu}g'(u)f^2(u)\,dt + f(u)\, dw.
  \end{equ}
\end{conjecture}

To test this conjecture we perform a numerical experiment, similar to
the one for conjecture~\ref{C3}; the result is shown in
figure~\ref{fig:multiplicative}.  The fit between predicted and
numerically determined correction term in
figure~\ref{fig:multiplicative} is worse than in
figure~\ref{fig:general} and thus the numerical test is not entirely
conclusive.

One possible reason is that the spatial resolution of our numerical
simulations may not be sufficient. Indeed, the argument of the
previous sections is based on a spatial averaging of the small-scale
fluctuations of the process.  In the case of multiplicative noise,
these small-scale fluctuations are themselves multiplied by a the
process $f(u)$, which is spatially quite rough. Therefore, this
spatial averaging will hold only on extremely small scales, where
$f(u)$ is essentially constant. In order to be seen by the numerical
simulation, these scales still need to be resolved at sufficient
precision to have some version of the law of large numbers.

\begin{figure}
  \begin{center}
    \includegraphics{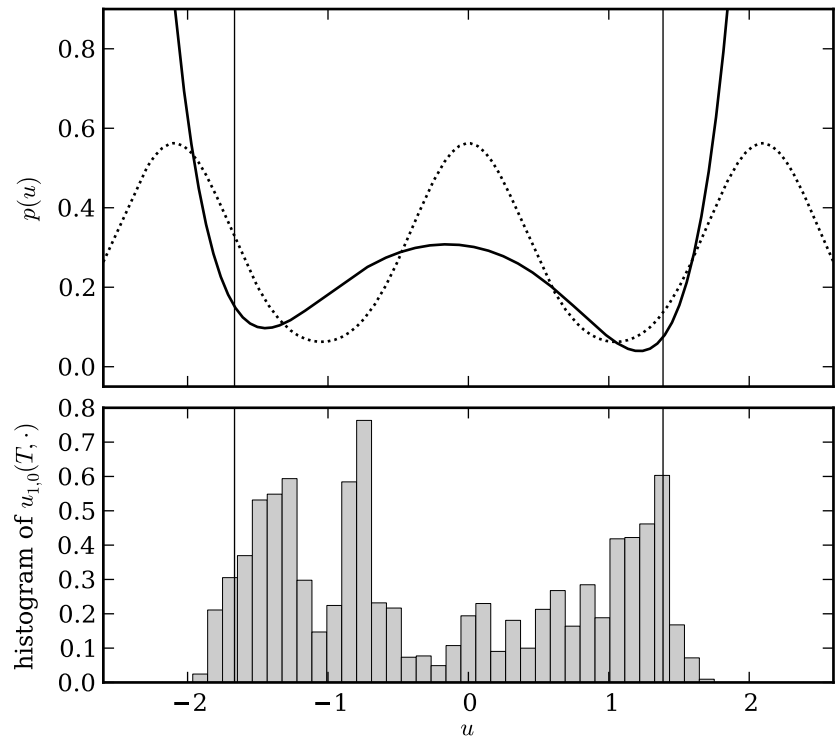}
  \end{center}
  \caption{\label{fig:multiplicative}\it Illustration of the
    convergence of \eref{e:C5.1} to~\eref{e:C5.2} in
    conjecture~\ref{C5}, for the case $g(u) = -u$ and $f(u) = 1 +
    \frac12 \cos(3u)$.  See figure~\ref{fig:general} for an
    explanation of the graphs.  Here we use a sixth-order polynomial
    $p$ to fit the correction term.  The figure shows that the fit is
    significantly worse than in figure~\ref{fig:general}.  See the
    text for a discussion of possible reasons for this effect.}
\end{figure}


\section{Small Noise/Viscosity Limit}
\label{S:viscosity limit}

One regime that is of particular interest is the small noise/small
viscosity limit.  If one takes $\nu \propto \sigma^2$ in
conjectures \ref{C1} and~\ref{C2}, one obtains a
non-vanishing correction term even for arbitrarily
small $\nu$ and $\sigma$!
It is therefore of interest to study approximations to
\begin{equ}[e:BurgEps]
  du = \eps\,\d_x^2 u\,dt - u \,\d_x u\,dt + \sqrt{\eps}\,dw
\end{equ}
for $\eps \ll 1$.

It is well-known that, in the limiting case $\eps=0$, finite
difference schemes for the Burgers equations can only be used with
extreme caution due to the presence of \textit{shocks} in the
solution.  These shocks are jump discontinuities of the solution.  The
values $u^-, u^+$ to the left/right of the jumps satisfy $u^- >
u^+$; in other words, the jumps are always downwards jumps.  For
viscosity solutions\footnote{Also called entropy solutions, these are
  the solutions that are obtained as limits of \eref{e:BurgEps} as
  $\eps \to 0$} to the inviscid Burgers equation, shocks move through
the system at velocity ${1\over 2}(u^+ + u^-)$.

What happens at the formation of a shock?  If the limiting non-viscous
Burgers equation is discretised as
\begin{equ}[e:consbad]
  \d_t u_n = - u_n \frac{u_{n+1} - u_n}{\delta},
\end{equ}
we expect the discretisation scheme to be stable only when the
discretisation is {\em upwind} in the sense that the direction of the
discretisation coincides with the direction of propagation of the
shock (see \textit{e.g.}\ \cite{Upwind,BookNA}), but even in this case
we expect the shock to propagate at the wrong speed.

Another problem is the following one: In the case of a
non-conservative discretisation of the type \eref{e:consbad}, a simple
linear stability analysis reveals that the discretised solutions
develop an {\em ultraviolet instability} (\textit{i.e.}\ the mode $v_n
= (-1)^n$ becomes unstable) in the regions where $u>0$.  Similarly,
the corresponding left-sided discretisation shows an instability in
the regions with~$u < 0$.  In contrast, for the case of a centred
discretisation, the highly oscillatory modes are stable independently
of the sign of~$u$.

\begin{remark}
  Conservative discretisations, \textit{i.e.}
  \begin{equ}
    \d_t u_n = -{1\over 2\delta}\bigl(u_{n+1}^2 - u_n^2\bigr),
  \end{equ}
  or variants thereof, still require the scheme to be upwind for
  stability, but in this case the shocks of the discretised system
  propagate at the correct speed.  Also, these schemes do not suffer
  from the ultraviolet instability.
\end{remark}

How is this picture modified for non-zero values of $\eps$?  The
instabilities discussed above grow at a speed $\CO(\delta^{-1})$ while
the stabilising effect of the viscosity is of the order $\eps
\delta^{-2}$.  Thus, one expects the viscosity to dominate only if
$\eps \gg \delta$.  Regarding the behaviour after the formation of a
shock, a simple boundary layer analysis shows that a typical shock for
\eref{e:BurgEps} has width $\CO(\eps)$, so that the caveats pointed
out above are expected to become relevant as soon as $\eps \lesssim
\delta$ (see for example \cite{EVan00} for a more sophisticated
boundary layer analysis that even goes to the next order in $\eps$).

On the other hand, at least away from shocks, the analysis performed
in section~\ref{sec:correction} holds as soon as $u$ can be
approximated by the solution to the linearised equation at
sufficiently small (but still much larger than $\delta$) spatial scales.
The $k$th mode presents spatial features on a lenghtscale of
order $k^{-1}$ and the nonlinearity essentially
propagates the solution through space
at speeds of order $1$.  Thus, the timescale on which the $k$th
mode varies due to the nonlinear effects is expected to be of
order~$k^{-1}$.  On the other hand, the relevant timescale of the
linear part for this mode is $(\eps k^2)^{-1}$ and we can conclude from this
heuristic consideration that the linearised equation is a good
approximation to the full solution for modes with $k \gg \eps^{-1}$,
\textit{i.e.}\ on space scales much smaller than~$\eps$.
Consequently, one again expects the results from
section~\ref{sec:correction} to be relevant as long as $\eps \gg
\delta$.  This leads to the following statements.

\begin{conjecture}\label{C6}
  For $\eps\ll 1$, we expect the solution to the finite difference
  approximation of \eref{e:BurgEps} to show the following behaviour.
  \begin{enumerate}
  \item For $\delta \ll \eps$, the discretised solution converges to
    the viscosity solution of
    \begin{equ}[e:BurgersCorr]
      \d_t u = -{1\over 2} \d_x (u^2) + {c\over 4},
    \end{equ}
    where $c \in \{1,0,-1\}$ depending on whether the discretisation
    is right-sided, centred, or left-sided.  Neither of the
    instabilities discussed above occur.
  \item For $\eps \ll \delta$, both viscosity and the noise term
    become irrelevant; the solution behaves like the corresponding
    approximation to the inviscid Burgers equation.  In particular, as
    long as there are no shocks and while solutions have the correct
    sign to prevent ultraviolet blow-up, we expect to converge to
    \eref{e:BurgersCorr} with $c=0$.  After the occurrence of a shock,
    one expects stability only if the scheme is upwind and, for
    discretisations of type~\eref{e:consbad}, shocks will have the
    wrong propagation speed.
  \end{enumerate}
\end{conjecture}

\begin{figure}
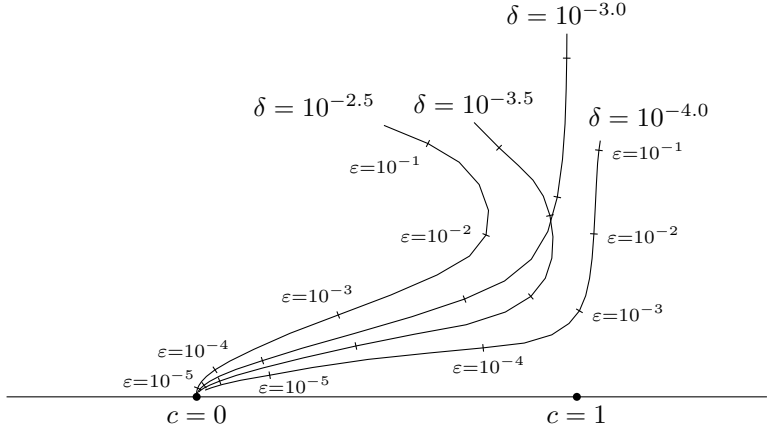

  \begin{center}
    \input fig7.tikz
  \end{center}
  \caption{\label{fig:small}\it Illustration of the limiting behaviour
    of a right-sided finite difference discretisation
    for~\eref{e:BurgEps} as $\eps\downarrow 0$ (for fixed $\delta$).
    Distances are shown in ``two-centre bipolar coordinates''
    as described in the text.
    One can see that, as $\eps$ gets small, the finite difference
    discretisation first gets close to the solution
    of~\eref{e:BurgersCorr} for $c=1$ but then finally converges to
    the solution with $c=0$ as $\eps$ gets smaller than~$\delta$.
    This behaviour reflects the dichotomy between the two cases of
    conjecture~\ref{C6}.}
\end{figure}

To test this conjecture we again perform a numerical experiment.
Keeping in line with the topic of this article, we focus on studying
how the presence of the extra $c/4$ term in~\eqref{e:BurgersCorr} is
affected by the values $\eps, \delta>0$.  For our experiment, we first
solve the limiting equation \eqref{e:BurgersCorr} numerically up to a
time $t>0$, once with $c=0$ and once with $c=1$, to get states
$u_0(t), u_1(t) \in L^2\bigl([0,2\pi], \R\bigr)$.  Now, for given
$\eps,\delta>0$, we solve the right-sided finite difference
discretisation~\eref{e:approxN} for~\eref{e:BurgEps} to get a solution
$u_{\eps,\delta}(t)$.  According to conjecture~\ref{C6} we expect
$u_{\eps,\delta}(t) \approx u_1$ for $\delta \ll \eps \ll 1$ and
$u_{\eps,\delta}(t) \approx u_0$ for $\eps \ll \delta$.

To verify this conjecture, in figure~\ref{fig:small} we consider, for
fixed $\delta$, the solution $u_{\eps,\delta}(t)$ as a function
of~$\eps$.  The coordinate system is chosen such that the (two
dimensional) distance of $u_{\eps,\delta}(t)$ to the point $c=0$ in
the graph equals $\bigl\|u_{\eps,\delta}(t) - u_0\|_{L^2}$; similarly
the distance of $u_{\eps,\delta}(t)$ to the point $c=1$ in the graph
equals $\bigl\|u_{\eps,\delta}(t) - u_1\|_{L^2}$ (this is sometimes
called ``two-centre bipolar coordinates'').  The simulation indicates
that the transition from $c = 1$ to $c=0$ takes place at around $\eps
\approx \delta$ as expected.


\begin{appendix}

\section{Simulations}
\label{S:simulate}

To verify the heuristic arguments presented above, the text utilises a
series of numerical results.  This appendix summarises some of the
technical aspects of these simulations.\footnote{For full details we
  refer to the source code of the programs used in these simulations,
  which is available for download at
  \url{http://seehuhn.de/programs/HairerVoss10} .}

We first describe how to implement the finite difference schemes used
to discretise SPDEs like \eref{e:Burgers}, \eref{e:BurgersGen}
and~\eref{e:multBurgers}: For the space discretisation of these
equations we approximate states $u\in L^2\bigl([0,2\pi], \R\bigr)$ by
vectors $u^N\in\R^N$.  The space discretisation of the differential
operators given by formula~\eref{e:burgersFD}, above.  The
finite difference discretisation of the white noise process $w$ is
$W^N/\sqrt{\Delta x}$, where $W^N$ is a standard Brownian motion on
$\R^N$ and $\Delta x = \delta = 2\pi/N$ is the space grid size.  This
leads to $\R^N$-valued stochastic differential equations of the form
\begin{equ}
  du^N = \nu L_N u^N \,dt + F_N(u^N)\,dt + \sigma(u^N)\,\sqrt{\frac{1}{\Delta x}} dW^N(t)
\end{equ}
where $L_N\in\R^{N\times N}$ is the discretisation of the linear part
and $F_N\colon \R^N\to \R^N$ is the discretisation of the
nonlinearity.

For discretising time we use the $\theta$-method
\begin{equs}
  u^{(n+1)}
  &= u^{(n)}
     + \nu L_N \bigl(\theta u^{(n+1)} + (1-\theta) u^{(n)}\bigr) \,\DT \\
    &\hskip15mm
     + F_N\bigl(u^{(n)}\bigr) \,\DT
     + \sigma(u^{(n)}) \sqrt{\frac{\DT}{\DX}} \,\xi^{(n)},
\end{equs}
where $\Delta t>0$ is the time step size, $u^{(n)}$ is the discretised
solution at time $n\,\Delta t$, the $\xi^{(n)}$ are i.i.d.,
$N$-dimensional standard normally distributed random variables, and
$\theta=1/2$.  Rearranging this equation gives
\begin{equs}[e:theta-step]
  \bigl( I - \nu \theta \DT L_N \bigr) u^{(n+1)}
  &= \bigl( I + \nu (1-\theta) \DT L_N \bigr) u^{(n)} \\
  &\hskip15mm
     + F_N(u^{(n)}) \,\DT
     + \sigma(u^{(n)}) \sqrt{\frac{\DT}{\DX}} \,\xi^{(n)}.
\end{equs}
Relation \eref{e:theta-step} allows to compute $u^{(n+1)}$ from
$u^{(n)}$; since $I - \nu \theta \DT L_N$ is cyclic tridiagonal, this
system can be solved efficiently.

Since we are using the partially implicit $\theta$-method for the
linear part of the SDE, there are no constraints on the time step size
$\DT$ arising from this term; on the other hand, since the non-linear
transport term is treated explicitly, one has to make sure that $v
\DT/\DX < C$, where $v$ is the largest speed of propagation appearing
in the solution and $C$ is the Courant number, thus ensuring that the
CFL condition is satisfied.  The resulting method can be used to
perform the simulations required to generate figures
\ref{fig:solutions}, \ref{fig:strange} and~\ref{fig:small}, as well as
the ``finite difference'' curve in figure~\ref{fig:burgers}.

As described in section~\ref{sec:correction}, we can compute a
spectral Galerkin approximation to the second order differential operator using
discrete Fourier transform.  This corresponds to replacing $L_N$
in~\eref{e:theta-step} with
\begin{equ}
  \tilde L_N = \CF^{-1} D \CF
\end{equ}
where $D = \diag\bigl(-0^2, -1^2, \ldots, -\lfloor N/2\rfloor^2\bigr)$
and $\CF\colon \R^N\to \C^{\lfloor N/2\rfloor+1}$ represents the discrete
Fourier transform (since the data is real, only $\lfloor N/2\rfloor+1$
of the Fourier coefficients need to be considered).  Because the
matrix $\tilde L_N$ is no longer tridiagonal, one should not try to
explicitly construct this matrix.  Instead one can use the fact that $\CF$
and $\CF^{-1}$ can be computed efficiently: we can compute the
right-hand side of \eref{e:theta-step} using
\begin{equ}
 \bigl( I + \nu (1-\theta) \DT L_N \bigr) u^{(n)}
 = \CF^{-1} \diag\Bigl( 1 - \nu (1-\theta) \DT\, k^2,\;
                       k=0, \ldots, \lfloor N/2\rfloor^2 \Bigr) \CF\, u^{(n)}
\end{equ}
and to solve \eref{e:theta-step} for $u^{(n+1)}$ we can use the
relation
\begin{equ}
 \bigl( I - \nu \theta \DT L_N \bigr)^{-1} b
 = \CF^{-1} \diag\Bigl( \frac{1}{1 + \nu \theta \DT\, k^2},\;
                       k=0, \ldots, \lfloor N/2\rfloor^2 \Bigr) \CF\, b.
\end{equ}
This technique allows to obtain the ``Galerkin'' curve in
figure~\ref{fig:burgers}.  Similarly, the rougher-than-white noise for
figure~\ref{fig:rough} was obtained by replacing the noise term
$\xi^{(n)}$ with
\begin{equ}
  \tilde\xi^{(n)} = \CF^{-1} \diag\bigl((1+k^2)^{-\delta}, \;
                      k=0, \ldots, \lfloor N/2\rfloor \bigr) \CF \, \xi^{(n)}.
\end{equ}

Finally, for the minimisation procedure performed to generate figures
\ref{fig:general} and~\ref{fig:multiplicative} we employ the simplex
algorithm by Nelder and Mead~\cite{NeMea65}.

\end{appendix}

\bibliographystyle{jochen}
\bibliography{./refs}
\end{document}

%% file: fig7.tikz
\begin{tikzpicture}[scale=5]
\fill (0,0) circle (.3pt) node[below] {$c=1$};
\fill (-1,0) circle (.3pt) node[below] {$c=0$};
\draw[thin] (-1.5,0) -- (0.5,0);
\draw (0.048154,0.655102) -- (0.068015,0.652751) node[right] {$\scriptstyle\varepsilon=10^{-1}$};
\draw (0.034863,0.432800) -- (0.054825,0.431567) node[right] {$\scriptstyle\varepsilon=10^{-2}$};
\draw (-0.001814,0.233233) -- (0.015358,0.222980) node[right] {$\scriptstyle\varepsilon=10^{-3}$};
\draw (-0.247441,0.139859) -- (-0.245343,0.119970) node[below=-2pt,xshift=2pt] {$\scriptstyle\varepsilon=10^{-4}$};
\draw (-0.809648,0.067201) -- (-0.806121,0.047514) node[right,xshift=-8px,yshift=-2px] {$\scriptstyle\varepsilon=10^{-5}$};
\draw (-0.977826,0.017763) -- (-0.961110,0.024057) -- (-0.932407,0.032787) -- (-0.884400,0.043893) -- (-0.807884,0.057357) -- (-0.695113,0.077808) -- (-0.548932,0.099320) -- (-0.389414,0.116162) -- (-0.246392,0.129914) -- (-0.138246,0.142648) -- (-0.065867,0.164475) -- (-0.020608,0.194579) -- (0.006772,0.228106) -- (0.023278,0.268074) -- (0.033467,0.314221) -- (0.040132,0.367303) -- (0.044844,0.432184) -- (0.048389,0.500914) -- (0.051433,0.564782) -- (0.054596,0.619635) -- (0.058085,0.653926) -- (0.061688,0.679544) node[above right,xshift=-8pt,yshift=2pt]{$\delta=10^{-4.0}$};
\draw (-0.211578,0.653479) -- (-0.197586,0.667770);
\draw (-0.080451,0.478673) -- (-0.060979,0.483239);
\draw (-0.127447,0.273849) -- (-0.114458,0.258641);
\draw (-0.583180,0.144681) -- (-0.578981,0.125127);
\draw (-0.943006,0.051117) -- (-0.936077,0.032356);
\draw (-0.994155,0.013130) -- (-0.989344,0.017476) -- (-0.980774,0.023263) -- (-0.965681,0.031039) -- (-0.939541,0.041736) -- (-0.895519,0.056954) -- (-0.824714,0.076942) -- (-0.719480,0.103349) -- (-0.581081,0.134904) -- (-0.428507,0.165830) -- (-0.291095,0.191245) -- (-0.187963,0.224916) -- (-0.120952,0.266245) -- (-0.082998,0.314565) -- (-0.065748,0.367790) -- (-0.062834,0.424416) -- (-0.070715,0.480956) -- (-0.088097,0.532154) -- (-0.115286,0.574667) -- (-0.153544,0.613373) -- (-0.204582,0.660625) -- (-0.270283,0.727671) node[above]{$\delta=10^{-3.5}$};
\draw (-0.036616,0.898346) -- (-0.016617,0.898094);
\draw (-0.061122,0.531815) -- (-0.041517,0.527857);
\draw (-0.298935,0.268631) -- (-0.291689,0.249990);
\draw (-0.830043,0.105550) -- (-0.823892,0.086519);
\draw (-0.986686,0.038695) -- (-0.974460,0.022867);
\draw (-0.998108,0.009662) -- (-0.996597,0.012905) -- (-0.993894,0.017240) -- (-0.989084,0.023035) -- (-0.980573,0.030781) -- (-0.965662,0.041126) -- (-0.939918,0.054860) -- (-0.896612,0.072730) -- (-0.826968,0.096034) -- (-0.723204,0.128771) -- (-0.586836,0.167585) -- (-0.433738,0.212557) -- (-0.295312,0.259311) -- (-0.190238,0.307204) -- (-0.119722,0.364978) -- (-0.076262,0.438656) -- (-0.051320,0.529836) -- (-0.037771,0.629299) -- (-0.031011,0.724161) -- (-0.028022,0.816959) -- (-0.026617,0.898220) -- (-0.026180,0.963114) node[above]{$\delta=10^{-3.0}$};
\draw (-0.387242,0.680780) -- (-0.396097,0.662847) node[below left,xshift=2pt] {$\scriptstyle\varepsilon=10^{-1}$};
\draw (-0.229537,0.425375) -- (-0.248020,0.433014) node[left] {$\scriptstyle\varepsilon=10^{-2}$};
\draw (-0.623975,0.207832) -- (-0.631128,0.226509) node[above left,xshift=10pt] {$\scriptstyle\varepsilon=10^{-3}$};
\draw (-0.946295,0.063307) -- (-0.957327,0.079989) node[above left,xshift=10pt] {$\scriptstyle\varepsilon=10^{-4}$};
\draw (-0.988075,0.019239) -- (-1.006630,0.026704) node[left,yshift=3pt,xshift=4pt] {$\scriptstyle\varepsilon=10^{-5}$};
\draw (-1.000414,0.007278) -- (-1.000323,0.009703) -- (-0.999986,0.012935) -- (-0.999146,0.017240) -- (-0.997352,0.022972) -- (-0.993774,0.030594) -- (-0.986947,0.040712) -- (-0.974346,0.054087) -- (-0.951811,0.071648) -- (-0.912998,0.094658) -- (-0.849583,0.126010) -- (-0.754214,0.168593) -- (-0.627551,0.217170) -- (-0.489258,0.270076) -- (-0.367444,0.323291) -- (-0.282382,0.373277) -- (-0.238779,0.429195) -- (-0.232333,0.494383) -- (-0.256936,0.563024) -- (-0.309570,0.622297) -- (-0.391669,0.671813) -- (-0.507342,0.719956) node[above left]{$\delta=10^{-2.5}$};
\end{tikzpicture}